\documentclass{article}[11pt,a4paper]

\usepackage[english,frenchb]{babel}
\usepackage{amsmath}
\usepackage{amsfonts}
\usepackage{amsthm}
\usepackage{amssymb}
\usepackage{bm}

\usepackage[latin1]{inputenc}

\usepackage{enumerate}

\usepackage[all,cmtip]{xy}

\newtheoremstyle{mystyle2}{}{}{}{2pt}{\scshape}{.}{ }{}
\newtheoremstyle{mystyle}{}{}{\slshape}{2pt}{\scshape}{.}{ }{}
\newtheoremstyle{etapestyle}{}{}{\itshape}{2em}{\sffamily}{:}{ }{\thmname{#1}}
\newtheoremstyle{definitionstyle}{}{}{}{2pt}{\bfseries}{.}{ }{}

\newtheorem{thm}{Th\'{e}or\`{e}me}[section]

\newtheorem{prop}[thm]{Proposition}

\newtheorem{lemme}[thm]{Lemme}

\theoremstyle{mystyle2}

\newtheorem{quest}[thm]{Question}
\theoremstyle{mystyle}

\theoremstyle{remark}

\theoremstyle{etapestyle}
\newtheorem{etape1}{Etape 1 }
\newtheorem{etape2}{Etape 2 }
\newtheorem{etape3}{Etape 3 }
\newtheorem{etape4}{Etape 4 }
\newtheorem{etape5}{Etape 5 }
\theoremstyle{definitionstyle}

\DeclareMathOperator{\codim}{codim} 
 \DeclareMathOperator{\Spec}{Spec}
\DeclareMathOperator{\Frac}{Frac} \DeclareMathOperator{\lon}{long}
\DeclareMathOperator{\prof}{prof}
\DeclareMathOperator{\coprof}{coprof}

\title{Le th\'{e}or\`{e}me de Bertini en famille}
\author{Olivier BENOIST}    
\date{}

\begin{document}

\maketitle

\begin{abstract} 

  On majore la dimension de l'ensemble des hypersurfaces de $\mathbb{P}^N$ dont l'intersection avec une variété projective intègre fixée n'est pas intègre. Les majorations obtenues sont optimales. Comme application, on construit, quand c'est possible, des hypersurfaces dont les intersections avec toutes les variétés d'une famille de variétés projectives intègres sont intègres. Le degré des hypersurfaces construites est explicite.

 \end{abstract}

\selectlanguage{english}

\begin{abstract} 

  We give upper bounds for the dimension of the set of hypersurfaces of $\mathbb{P}^N$ whose intersection with a fixed integral projective variety is not integral. Our upper bounds are optimal. As an application, we construct, when possible, hypersurfaces whose intersections with all the varieties of a family of integral projective varieties are integral. The degree of the hypersurfaces we construct is explicit.

 \end{abstract}

\selectlanguage{frenchb}

\vspace{1em}




On fixe $K$ un corps de caractéristique quelconque, qui sera
systématiquement sous-entendu. Par exemple,
$\mathbb{P}^N=\mathbb{P}^N_K$. Par sous-variété de $\mathbb{P}^N$, on entendra sous-schéma fermé géométriquement intègre de $\mathbb{P}^N$.
On note $\mathcal{H}_e=\mathcal{H}_{e,N}=\mathbb{P}(H^0(\mathbb{P}^N,\mathcal{O}(e)))$
l'espace des hypersurfaces de degré $e$ de $\mathbb{P}^N$.

Si $X\subset\mathbb{P}^N$ est une sous-variété, on notera $\mathcal{F}^{\rm{int}}_e(X)$ (resp. $\mathcal{F}^{\rm{igr}}_e(X)$) le sous-ensemble de $\mathcal{H}_e$ constitué des hypersurfaces dont l'intersection avec $X$ n'est pas géométriquement intègre de codimension $1$ dans $X$ (resp. n'est pas géométriquement irréductible et génériquement réduite de codimension $1$ dans $X$). Le théorème de Bertini (voir par exemple \cite{Jouanolou} I 6.10) montre que si $\dim(X)\geq 2$, $\mathcal{F}^{\rm{int}}_e(X)$ est strictement inclus dans $\mathcal{H}_e$. 

 Pour les questions que nous allons étudier, la propriété \og irréductible et génériquement réduit\fg \hspace{0.2mm} se comporte mieux  que l'intégrité. D'autre part, pour faire fonctionner des arguments de déformation, nous aurons besoin de savoir que ces mauvais lieux sont fermés. C'est pourquoi nous allons utiliser la variante suivante du théorème de Bertini, qui fait l'objet de la première partie de cet article. 

\begin{thm}\label{bertinigr}

Soit $X$ une sous-variété de $\mathbb{P}^N$ de dimension $\geq2$. Alors  $\mathcal{F}^{\rm{int}}_e(X)$ et $\mathcal{F}^{\rm{igr}}_e(X)$ sont des fermés stricts de $\mathcal{H}_e$.

 \end{thm}

 Quand $K$ est infini, une conséquence de ce théorème est qu'il existe des hypersurfaces dans le complémentaire de $\mathcal{F}^{\rm{int}}_e(X)$ (resp. $\mathcal{F}^{\rm{igr}}_e(X)$). On cherche dans cet article à
obtenir une version \og en famille\fg \hspace{0.2mm} de ce théorème.
On se pose plus précisément la question suivante :

\begin{quest}\label{quest0}

\'{E}tant donnée une famille de sous-variétés de dimension
$\geq2$ de $\mathbb{P}^N$, peut-on trouver une hypersurface dont les
intersections avec toutes les variétés de cette famille soient géométriquement
intègres (resp. géométriquement irréductibles et génériquement réduites) ? Si la réponse est positive, peut-on trouver une telle
hypersurface de petit degré ?

\end{quest}

En fait, cela revient à contrôler, pour chaque variété
$X$ de la famille, la codimension de $\mathcal{F}^{\rm{int}}_e(X)$ (resp. $\mathcal{F}^{\rm{igr}}_e(X)$) dans $\mathcal{H}_e$.
Plus précisément, il faut répondre à la question suivante :

\begin{quest}\label{quest}

La codimension de $\mathcal{F}^{\rm{int}}_e(X)$ (resp. $\mathcal{F}^{\rm{igr}}_e(X)$) dans $\mathcal{H}_e$ tend-elle vers l'infini avec $e$ ?

\end{quest}

  Dans la deuxième partie de cet article, on obtient des minorations optimales de $\codim_{\mathcal{H}_e}(\mathcal{F}^{\rm{igr}}_e(X))$ en fonction de $e$ et de $\dim(X)$. En particulier, la réponse à la question \ref{quest} pour $\mathcal{F}^{\rm{igr}}_e(X)$ est positive. L'énoncé est le suivant :

\begin{thm}\label{codimigr}

Soit $X$ une sous-variété de $\mathbb{P}^N$ de dimension $n\geq2$. Alors :

$$\codim_{\mathcal{H}_e}(\mathcal{F}^{\rm{igr}}_e(X))\geq
\begin{cases}
\ n-1        &\mbox{si  \hspace{1em}} e=1
\\\ \binom{e+n-1}{e}-n &\mbox{si  \hspace{1em}} e\geq2.

\end{cases}$$

De plus, ces bornes sont optimales. Elles sont atteintes pour un cône sur une courbe quand $e=1$ et pour un espace linéaire quand $e\geq2$.

 \end{thm}

  On en déduit dans la troisième partie des minorations analogues pour $\mathcal{F}^{\rm{int}}_e(X)$. L'énoncé qui suit montre que la réponse à la question \ref{quest} pour $\mathcal{F}^{\rm{int}}_e(X)$ est positive si et seulement si $X$ n'a pas de point fermé de profondeur $1$.

\begin{thm}\label{codimint}

Soit $X$ une sous-variété de $\mathbb{P}^N$ de dimension $n\geq2$. Alors :

\begin{enumerate}[(i)]
\item Si $X$ n'a pas de point de profondeur $1$ et de codimension $>1$,  $$\codim_{\mathcal{H}_e}(\mathcal{F}^{\rm{int}}_e(X))\geq
\begin{cases}
\ n-1        &\mbox{si  \hspace{1em}} e=1
\\\ \binom{e+n-1}{e}-n &\mbox{si  \hspace{1em}} e\geq2.
\end{cases}$$

\item Si $X$ n'a pas de point fermé de profondeur $1$,  $$\codim_{\mathcal{H}_e}(\mathcal{F}^{\rm{int}}_e(X))\geq
\begin{cases}
\ e+1        &\mbox{si  \hspace{1em}} n\geq3
\\ e-1        &\mbox{si  \hspace{1em}} n=2 \text{ et } e\geq2
\\\ 1        &\mbox{si  \hspace{1em}} n=2 \text{ et }  e=1.
\end{cases}$$

\item Si $X$ possède un point fermé de profondeur $1$, $$\codim_{\mathcal{H}_e}(\mathcal{F}^{\rm{int}}_e(X))=1 \text{ pour tout } e.$$

De plus, ces bornes sont optimales.
\end{enumerate}

 \end{thm}

Remarquons que la condition \og ne pas posséder de point fermé de
profondeur $1$ et de codimension $>1$\fg \hspace{0.2mm} est la condition $S2$ de Serre. Elle est en particulier vérifiée pour
les variétés normales et les variétés de Cohen-Macaulay.

\vspace{1em}

  Enfin, dans une quatrième partie, on obtient le théorème suivant répon\-dant à la question \ref{quest0}. La restriction sur les fibres de la famille, dans le cas intègre, est nécessaire au vu du théorème précédent.

\begin{thm}\label{famillebert}

Soit une famille plate de sous-variétés de dimension $n\geq2$
de $\mathbb{P}^N$, c'est-à-dire un diagramme commutatif de
$K$-schémas de type fini  :$$\xymatrix{
\mathfrak{X}\ar[r]^{i\hspace{1em}}\ar[d]_{\pi}& \mathbb{P}^N\times V \ar[ld]^{pr_2}                          \\
V   }$$où $\pi$ est plat à fibres géométriquement intègres de dimension $n\geq2$ et
$i$ est une immersion fermée. Alors :

\begin{enumerate}[(i)]
\item Soit $e$ tel que $\binom{e+n-1}{e}-n-1\geq\dim(V)$. L'ensemble des hypersurface $H$ de $\mathbb{P}^N$ de degré
$e$ telles que pour tout $v\in V$, $\mathfrak{X}_v\cap H$ est géométriquement irréductible et génériquement réduit de codimension $1$ dans $\mathfrak{X}_v$ contient un ouvert non vide.

\item Soit $e\geq\dim(V)+2$. Supposons que les $\mathfrak{X}_v$ n'ont pas de points fermés de profondeur
$1$. Alors l'ensemble des hypersurfaces $H$ de $\mathbb{P}^N$ de degré
$e$ telles que pour tout $v\in V$, $\mathfrak{X}_v\cap H$ est géométriquement intègre de codimension $1$ dans $\mathfrak{X}_v$ contient un ouvert non vide.
\end{enumerate}
Si $V$ est propre, ces ensembles sont des ouverts non vides. 
Enfin, quand le corps $K$ est infini, on peut trouver une telle hypersurface $H$ définie sur $K$.

 \end{thm}

Ces questions ont été motivées par les constructions de
\cite{Debarreart1} de variétés dont le fibré cotangent est ample. En
particulier, le théorème \ref{codimint}
corrige et précise Lemma 12 de {\it loc. cit.}, erroné. 

Quand $K$ est fini, les méthodes utilisées ici pour prouver le théorème \ref{famillebert} ne permettent pas de construire une hypersurface définie sur $K$. L'analogue de cette question pour la version lisse du théorème de Bertini a été étudiée et résolue par Poonen dans \cite{Poonen}.

\vspace{1em}

Je tiens à remercier chaleureusement O. Debarre pour les nombreux conseils qu'il m'a donnés.

\section{Théorème de Bertini}

On va démontrer dans cette partie le théorème \ref{bertinigr}.

\subsection{Ouverture de la propriété \og irréductible et géné\-ri\-que\-ment réduit\fg}

Pour cela, on va commencer par déterminer des conditions sur une famille de schémas sous lesquelles l'ensemble des fibres qui sont irréductibles et géné\-ri\-que\-ment réduites est ouvert : c'est le rôle de la proposition \ref{ouv}.
 C'est une proposition très proche d'énoncés de \cite{EGA43},
et les démonstrations sont calquées sur celles s'y trouvant. Par
conséquent, on multipliera les références à cet ouvrage.

\begin{prop}\label{ouv}

Soit $f:X\rightarrow S$ un morphisme de schémas propre, plat et de présentation finie. On suppose que les composantes irréductibles des fibres de $f$ sont toutes de la même dimension $n$.
Alors l'ensemble $E$ des $s\in S$ tels que $X_s$
est géométriquement irréductible et génériquement réduit est ouvert dans $S$.

\end{prop}

\begin{proof}[$\mathbf{Preuve}$]

On scinde la preuve en plusieurs étapes. 

\begin{etape1}

Sous les hypothèses de l'énoncé, si $S$ est le spectre d'un anneau
de valuation discrète de point générique $s'$ et de point fermé $s$, si de plus les composantes irréductibles de
$X_{s'}$ sont géométriquement irréductibles, alors si
$s\in E$, on a également $s'\in E$.

\end{etape1}

Montrons d'abord que $X_{s'}$ est géométriquement génériquement
réduit. Par \cite{EGA43} 12.1.1 (vii), l'ensemble $U$ des $x\in
X$ tels que $X_{f(x)}$ est géométriquement réduit en $x$ est ouvert
dans $X$. Son complémentaire $F$ est donc fermé. Par hypothèse, $F_s$ est
de dimension $<n$. Or, par propreté de $f$, la dimension des fibres de $f:F\rightarrow S$ est semi-continue supérieurement, donc $\dim(F_{s'})<n$, et $X_{s'}$
est bien géométriquement génériquement réduit.

Supposons ensuite par l'absurde que $X_{s'}$ ne soit pas irréductible, et soient $\eta_1$ et $\eta_2$ deux points maximaux distincts de $X_{s'}$. Par propreté, la dimension des fibres de $f :
\overline{\{\eta_i\}}\rightarrow S$ est semi-continue supérieurement. On peut donc choisir un point maximal $z_i$ de $\overline{\{\eta_i\}}_s$ dont l'adhérence est de dimension $\geq n$. Comme
les composantes irréductibles de $X_s$ sont de dimension $n$ par hypothèse, $z_i$
est nécessairement un point maximal de $X_s$.

Par \cite{EGA42} 2.3.4, qui s'applique par platitude, on voit que
les points maximaux de $X$ sont exactement ceux de $X_{s'}$. Cela
permet de vérifier les hypothèses de \cite{EGA42} 3.4.1.1 et de
montrer que si $z_1$ et $z_2$ coïncidaient, on aurait
$$\lon((\mathcal{O}_{X_s})_{z_1})\geq
\lon((\mathcal{O}_{X})_{\eta_1})+\lon((\mathcal{O}_{X})_{\eta_2})\ge
2,$$ ce qui est impossible car $X_s$ est génériquement réduit.
Ainsi, $z_1\neq z_2$, donc $X_s$ n'est pas irréductible, ce qui est absurde.

\begin{etape2}

Sous les hypothèses de l'énoncé, si $S$ est affine noethérien
intègre de point générique $s'$, si de plus les composantes
irréductibles de $X_{s'}$ sont géomé\-tri\-que\-ment
irréductibles, et si $E$ est non vide, alors $s'\in E$.

\end{etape2}

Supposons que $s\in E$ est différent de $s'$. Par \cite{EGA2} 7.1.7, 
qui s'applique par noethé\-ria\-ni\-té, on peut trouver un schéma
$T$, spectre d'un anneau de valuation discrète de point fermé $t$ et
de point générique $t'$, et un morphisme $g :T\rightarrow S$ tel que
$g(t)=s$ et $g(t')=s'$. On peut alors appliquer l'étape $1$ après
changement de base par $g$, ce qui montre $s'\in E$.

\begin{etape3}

Sous les hypothèses de l'énoncé, si $S=\Spec(A)$ est affine
noethé\-rien, $E$ est stable par générisation.

\end{etape3}

On se donne $s\in E$ et $s'$ une générisation de $s$ : on veut
montrer que $s'\in E$. Quitte à remplacer $S$ par $\overline{\{s'\}}$,
on peut supposer $S$ intègre de point générique $s'$.

Par \cite{EGA42} 4.6.8, il existe une extension finie $\tilde{K}$
de $K=\Frac(A)$ telle que les composantes irréductibles de $(X_{s'})_{\tilde{K}}$ soient
géométriquement irréductibles. Comme il existe une base de
$\tilde{K}$ sur $K$ formée d'éléments entiers sur $A$, l'anneau
$\tilde{A}$ engendré par ces éléments est fini sur $A$ et est de
corps de fractions $\tilde{K}$. On pose
$\tilde{S}=\Spec(\tilde{A})$, de sorte que le morphisme $g
:\tilde{S}\rightarrow S$ est surjectif car fini et dominant.

On peut alors appliquer l'étape $2$ après avoir changé de base
par $g$. Notons $\tilde{E}$ l'ensemble évident.  On a
$\tilde{E}=g^{-1}(E)$, donc $E=g(\tilde{E})$ par surjectivité de
$g$. Comme $\tilde{E}$ contient le point générique de $\tilde{S}$,
$E$ contient bien le point générique de $S$, ce qui conclut.

\begin{etape4}

Sous les hypothèses de l'énoncé, $E$ est ouvert.

\end{etape4}

Comme la proposition est locale sur $S$, on peut supposer que
$S=\Spec(A)$ est affine. 

Quand $S$ est noethérien, on va montrer que $E$ est ouvert en appliquant le critère \cite{EGA3} 
0.9.2.5. Vérifions-en les hypothèses. D'une part, $E$ est 
stable par générisation par l'étape $3$. D'autre part, $E$ est localement constructible par \cite{EGA43} 9.7.7 et 9.8.7. 

Si $A$ n'est pas noethérien, on écrit $A$ comme limite inductive de
ses sous-anneaux de type fini sur $\mathbb{Z}$. Alors, par
\cite{EGA43} 8.9.1, 8.10.5 (xii) et 11.2.6, on peut trouver un
sous-anneau noethérien $A_0$ de $A$ et un morphisme $f_0 : X_0\rightarrow\Spec(A_0)$ vérifiant les même propriétés que $f$, et tel que
$f$ soit obtenu à partir de $f_0$ par
extension des scalaires. Si $g:\Spec(A)\rightarrow \Spec(A_0)$ est
le changement de base et si $E_0$ est l'ensemble évident, on a
$E=g^{-1}(E_0)$. Comme $E_0$ est ouvert par le cas noethérien, $E$
est ouvert. D'où le résultat.

\end{proof}

\subsection{Démonstration du théorème}\label{preuvebertinigr}

Démontrons maintenant la version dont nous aurons besoin du théorème de Bertini, c'est-à-dire le théorème \ref{bertinigr}.

 \begin{proof}[$\mathbf{Preuve \text{ }du \text{ }th\acute{e}or\grave{e}me}$]

  Par \cite{Jouanolou} I 6.10 appliqué à l'inclusion de $X$ dans $\mathbb{P}^N$, $\mathcal{F}^{\rm{int}}_e(X)$ est strictement inclus dans $\mathcal{H}_e$. Comme $\mathcal{F}^{\rm{igr}}_e(X)\subset\mathcal{F}^{\rm{int}}_e(X)$, il suffit de montrer que $\mathcal{F}^{\rm{int}}_e(X)$ et $\mathcal{F}^{\rm{igr}}_e(X)$ sont fermés dans $\mathcal{H}_e$. On fait la preuve pour $\mathcal{F}^{\rm{igr}}_e(X)$ ; la fermeture de $\mathcal{F}^{\rm{int}}_e(X)$ se montre de la même manière en utilisant \cite{EGA43} 12.2.4 (viii) à la place de la proposition \ref{ouv}.

On introduit $\mathfrak{Z}\subset\mathbb{P}^N\times\mathcal{H}_e$ la
famille paramétrant les $\mathfrak{Z}_{H}=X\cap H$ et $q :
\mathfrak{Z}\rightarrow\mathcal{H}_e$ la projection canonique. On note $\mathcal{U}$ l'ouvert 
de $\mathcal{H}_e$ où $X\nsubseteq
 H$. Comme le polynôme de Hilbert de ses fibres est constant, en appliquant le critère III 9.9 de \cite{Hartshorne}, on voit que
 $q_\mathcal{U}:q^{-1}(\mathcal{U})\rightarrow \mathcal{U}$ est plat.
Soit alors $\mathcal{U}'$ le sous-ensemble de $\mathcal{U}$ constitué des $H$ tels
que $X\cap H$ est géométriquement irréductible et génériquement réduit. Il est ouvert par la proposition  \ref{ouv}. Par conséquent, $\mathcal{F}^{\rm{igr}}_e(X)$, qui est le complémentaire de $\mathcal{U}'$ dans $\mathcal{H}_e$, est fermé.

\end{proof}

\section{Minoration de la codimension de $\mathcal{F}^{\rm{igr}}_e(X)$}

Cette partie contient la démonstration du théorème \ref{codimigr}. Comme celui-ci est insensible à l'extension des scalaires, on supposera que $K$ est algébriquement clos.

L'idée d'utiliser une dégénérescence de $X$ vers une réunion d'espaces linéaires pour montrer ce théorème est due à Zak.

\subsection{Optimalité des minorants}

On montre dans ce paragraphe que les minorations du théorème \ref{codimigr} sont optimales.

\begin{prop}

Soit $n\geq2$ et $e\geq1$. Alors on peut trouver une sous-variété $X$ d'un espace projectif $\mathbb{P}^N$ pour laquelle l'inégalité du théorème \ref{codimigr} est une égalité.

\end{prop}

\begin{proof}[$\mathbf{Preuve}$]

Si $e\geq2$, le lemme \ref{caslineaire} ci-dessous montre qu'on peut prendre pour $X$ un sous-espace linéaire de dimension $n$ de $\mathbb{P}^N$.

Si $e=1$, on choisit pour $X$ un cône de base une courbe intègre de degré $\geq2$ et de sommet un sous-espace linéaire $L$ de dimension $n-2$ de $\mathbb{P}^N$. Notons $\mathcal{G}$ l'ensemble des éléments de $\mathcal{H}_{1}$ contenant $L$. Si $H\in \mathcal{G}$,  $X\cap H$ est réunion de sous-espaces linéaires de $\mathbb{P}^N$ et ne peut donc être irréductible et génériquement réduit pour raison de degré. Par conséquent, $\mathcal{G}\subset\mathcal{F}^{\rm{igr}}_1(X)$. D'autre part, par le lemme \ref{codimhyper} ci-dessous, $\codim_{\mathcal{H}_1}(\mathcal{G})=n-1$. On en déduit $\codim_{\mathcal{H}_1}(\mathcal{F}^{\rm{igr}}_1(X))\leq n-1$ et donc que l'inégalité de \ref{codimigr} est une égalité.

\end{proof}

\begin{lemme}\label{caslineaire}
Soit $X$ un sous-espace linéaire de dimension $n$ de $\mathbb{P}^N$, et $e\geq2$. Alors :

$$\codim_{\mathcal{H}_{e}}(\mathcal{F}^{\rm{igr}}_e(X))={e+n-1 \choose e}-n$$
\end{lemme}

\begin{proof}

On est immédiatement ramenés au cas où $N=n$ et $X=\mathbb{P}^n$.
Les points fermés de $\mathcal{F}^{\rm{igr}}_e(\mathbb{P}^n)$ correspondent aux
hypersurfaces dont une équation n'est pas irréductible, donc est
réunion de sous-ensembles correspondant aux degrés $k$ et $e-k$ des
deux facteurs d'une décomposition. Ceci permet de calculer :

$$\begin{array}{llcl} \codim_{\mathcal{H}_{e}}(\mathcal{F}^{\rm{igr}}_e(\mathbb{P}^n))&= & \min_{1\leq k\leq e-1}\left({e+n \choose
n}-{k+n \choose n}-{e-k+n \choose
n}+1\right)\\

&=& \min_{k\in\{1,e-1\}}\left({e+n \choose n}-{k+n \choose n}-{e-k+n
\choose
n}+1\right)\\

&=&{e+n \choose n}-n-1-{e-1+n \choose
n}+1\\

&=&{e+n-1 \choose e}-n.\\

\end{array}$$

\end{proof}

\begin{lemme}\label{codimhyper}

Soit $X$ une sous-variété de $\mathbb{P}^N$ de dimension $n$.
Notons $\mathcal{G}$ l'ensemble des hypersurfaces de degré $e$ contenant $X$.
On a $$\codim_{\mathcal{H}_e}(\mathcal{G})\geq\binom{e+n}{e}.$$
De plus, quand $X$ est un sous-espace linéaire de $\mathbb{P}^N$, on a égalité.

\end{lemme}

\begin{proof}[$\mathbf{Preuve}$]

Soit $L$ un sous-espace linéaire de dimension $N-n-1$ de
$\mathbb{P}^N$ ne rencontrant pas $X$. Les fibres de la projection
$\pi_L:X\rightarrow\mathbb{P}^n$ depuis $L$ sont toutes finies non
vides. Ceci montre que $\mathcal{G}$ et l'ensemble $\mathcal{C}$ des cônes de sommet $L$
ne s'intersectent pas dans $\mathcal{H}_e$. On peut alors appliquer
le théorème de l'intersection projective :
$$\codim_{\mathcal{H}_e}(\mathcal{G})\geq \dim(\mathcal{C})+1=\binom{e+n}{e}.$$

Finalement, quand $X$ est un sous-espace linéaire de $\mathbb{P}^N$, $X$ et $L$ sont supplé\-men\-taires dans $\mathbb{P}^N$, et on voit facilement que $\mathcal{G}$ et $\mathcal{C}$ sont des sous-espaces linéaires supplémentaires dans $\mathcal{H}_e$. On a alors $$\codim_{\mathcal{H}_e}(\mathcal{G})= \dim(\mathcal{C})+1=\binom{e+n}{e}.$$

\end{proof}

\subsection{Réduction au cas d'une hypersurface qui est un cône sur une courbe plane}

Dans ce paragraphe, on prépare la preuve du théorème \ref{codimigr} en effectuant un certain nombre de réductions. On commence par se ramener par projection au cas où $X$ est une hypersurface.

\begin{prop}\label{reduc1}

Le théorème \ref{codimigr} se déduit du cas particulier où $X$ est une hypersurface.

\end{prop}

\begin{proof}[$\mathbf{Preuve}$]

Si $X=\mathbb{P}^N$, c'est le lemme \ref{caslineaire} quand $e\geq2$ ou le lemme \ref{codimhyper} quand $e=1$.
Si $X$ est une hypersurface, il n'y a rien à démontrer. 

Si $X$ est de codimension $\geq2$ dans $\mathbb{P}^N$, on choisit, par le théorème de Bertini lisse, un sous-espace linéaire $L_1$ de dimension $N-n$ de $\mathbb{P}^N$ dont l'intersection avec $X$ est constituée de points réduits $P_i$. Soit $L_2$ un hyperplan de $L_1$ contenant $P_1$ et aucun des $P_i$, $i>1$, et $L$ un hyperplan de $L_2$ ne contenant pas $P_1$. On note $\pi:\mathbb{P}^N\rightarrow \mathbb{P}^{n+1}$ la projection depuis $L$ et $Q=\pi(P_1)$. La variété $Y=\pi(X)$ munie de sa structure réduite est une hypersurface de $\mathbb{P}^{n+1}$ ; on considère $\pi|_X:X\rightarrow Y$. Par construction, $\pi|_X^{-1}(Q)$ est le point réduit $P_1$. Par conséquent, comme $\pi|_X$ est propre, $\pi|_X$ est génériquement fini de degré $1$, c'est-à-dire birationnel. On notera $E$ le fermé strict de $Y$ au-dessus duquel $\pi|_X$ n'est pas un isomorphisme.

Soit $\mathcal{C}$ le fermé de $\mathcal{H}_{e,N}$ constitué des
cônes de sommet $L$. Par théorème de l'intersection projective,
$$\codim_{\mathcal{H}_{e,N}}(\mathcal{F}^{igr}_e(X))\geq
\codim_\mathcal{C}(\mathcal{F}^{igr}_e(X)\cap \mathcal{C}).$$

 On
vérifie en comparant les diviseurs de Cartier que
$\pi^*\mathcal{O}_{\mathbb{P}^{n+1}}(e)=\mathcal{O}_{\mathbb{P}^{N}\setminus\{L\}}(e)$,
et que tirer en arrière les sections globales induit une bijection
$\pi^*$ entre $\mathcal{H}_{e,n+1}$ et $\mathcal{C}$. Moyennant
cette identification entre $\mathcal{H}_{e,n+1}$ et $\mathcal{C}$,
on va montrer que $$(\mathcal{F}^{igr}_e(X)\cap
\mathcal{C})\subset\mathcal{F}^{igr}_e(Y)\cup \mathcal{G} ,$$ où $\mathcal{G}$ désigne
l'ensemble des hypersurfaces contenant une composante irréductible de $E$ de codimension
$1$ dans $Y$. On pourra alors conclure en appliquant l'hypothèse à $Y$ d'une part, et le lemme \ref{codimhyper} d'autre part.

Pour cela, soit $H\notin \mathcal{F}^{igr}_e(Y)\cup \mathcal{G}$. Comme
$H\notin \mathcal{G}$, $(Y\cap H)\setminus E$ est dense dans  $(Y\cap H)$, et $\pi|_X$ étant surjectif,
$(X\cap \pi^*H)\setminus \pi|_X^{-1}(E)$ est dense dans  $(X\cap \pi^*H)$. Comme $S\notin \mathcal{F}^{igr}_e(Y)$, $(Y\cap H)\setminus E$ est irréductible et génériquement réduit. C'est donc aussi le cas de $(X\cap \pi^*H)\setminus \pi|_X^{-1}(E)$ qui lui est isomorphe, et de $(X\cap \pi^*H)$ par densité. On a donc bien $H\notin (\mathcal{F}^{igr}_e(X)\cap
\mathcal{C})$.

\end{proof}

Un argument de déformation permet ensuite d'effectuer la réduction suivante :

\begin{prop}\label{reduc2}

Le théorème \ref{codimigr} se déduit du cas particulier où $X$ est une hypersurface qui est un cône sur une courbe plane.

\end{prop}

\begin{proof}[$\mathbf{Preuve}$]

Par \ref{reduc1}, on peut supposer que $X$ est une hypersurface. Par Bertini, on choisit des coordonnées dans lesquelles elle est d'équation $G(x_0,\ldots,x_N)=0$ de sorte que la courbe $G(x_0,x_1,x_2,0,\ldots,0)=0$ soit intègre.
Soit $\mathfrak{X}$ la sous-variété de $\mathbb{P}^{N}\times\mathbb{A}^{1}$ d'équation $G(x_0,x_1,x_2,tx_3,\ldots,tx_N)=0$. C'est une famille plate d'hypersurfaces de $\mathbb{P}^{N}$, triviale de fibre $X$ au-dessus de $\mathbb{A}^{1}\setminus\{0\}$, et telle que $\mathfrak{X}_0$ est un cône sur une courbe plane intègre.

Soit $\mathcal{F}$ le sous-ensemble de $\mathbb{A}^{1}\times\mathcal{H}_e$ constitué des $(t,H)$ tels que $\mathfrak{X}_t\cap H$ n'est pas irréductible et génériquement réduit de codimension $1$ dans $\mathfrak{X}_t$. Comme dans la preuve en \ref{preuvebertinigr} du théorème \ref{bertinigr}, on montre que $\mathcal{F}$ est fermé dans $\mathbb{A}^{1}\times\mathcal{H}_e$.

Les fibres de $p_1 :
\mathcal{F}\rightarrow \mathbb{A}^{1}$ sont les $p_1^{-1}(t)=\mathcal{F}^{\rm{igr}}_e(\mathfrak{X}_t)$. Par fermeture de $\mathcal{F}$, $p_1$ est propre, donc la dimension de ses fibres est semi-continue supérieurement. Ainsi, il existe $t\not=0$ tel que $\dim(\mathcal{F}^{\rm{igr}}_e(\mathfrak{X}_t))\leq\dim(\mathcal{F}^{\rm{igr}}_e(\mathfrak{X}_0))$. Comme $\mathfrak{X}_t$ est projectivement équivalente à $X$ et que l'on sait majorer la dimension de $\mathcal{F}^{\rm{igr}}_e(\mathfrak{X}_0)$ par hypothèse, on peut conclure.

\end{proof}

\subsection{Fin de la démonstration}

Achevons de démontrer le théorème \ref{codimigr}. On va adopter la même stratégie qu'en \ref{reduc2}, en utilisant maintenant une déformation de $X$ en une réunion d'hyperplans.

\begin{proof}[$\mathbf{Preuve \text{ }du \text{ }th\acute{e}or\grave{e}me\text{ } \ref{codimigr}}$]

\begin{etape1}
Construction d'une déformation.
\end{etape1}

\vspace{1em}

Par la proposition \ref{reduc2}, on suppose que $X$ est une hypersurface d'équation $G(x_0,x_1,x_2)=0$, qui est un cône de sommet l'espace linéaire $S$ d'équations $x_0=x_1=x_2=0$ sur une courbe plane intègre $C$ de degré $d$.
Par Bertini, on choisit nos coordonnées de sorte que $G(x_0,x_1,0)=0$ soit constitué de $d$ points réduits.
Soit $\mathfrak{X}$ la sous-variété de $\mathbb{P}^{N}\times\mathbb{A}^{1}$ d'équation $G(x_0,x_1,tx_2)=0$. C'est une famille plate d'hypersurfaces de $\mathbb{P}^{N}$, triviale de fibre $X$ au-dessus de $\mathbb{A}^{1}\setminus\{0\}$, et telle que $\mathfrak{X}_0$ est réunion réduite de $d$ hyperplans distincts $L_1,\ldots, L_d$ s'intersectant le long d'un espace linéaire commun $L : x_0=x_1=0$.

On note $\mathfrak{Z}\subset\mathbb{P}^N\times\mathbb{A}^{1}\times\mathcal{H}_e$ la famille paramétrant les $\mathfrak{Z}_{t,H}=\mathfrak{X}_t\cap H$ et $q :
\mathfrak{Z}\rightarrow\mathbb{A}^{1}\times\mathcal{H}_e$ la projection canonique. Soit $\mathcal{F}$ le sous-ensemble de $\mathbb{A}^{1}\times\mathcal{H}_e$ constitué des $(t,H)$ tels que $\mathfrak{Z}_{t,H}$ n'est pas irréductible et génériquement réduit de codimension $1$ dans $\mathfrak{X}_t$. 
On note $\mathcal{U}$ l'ouvert
de $\mathbb{A}^{1}\times\mathcal{H}_e$ où $\mathfrak{X}_t\cap
 H$ est de codimension $1$ dans $\mathfrak{X}_t$. Comme le polynôme de Hilbert de ses fibres est constant, en appliquant le critère III 9.9 de \cite{Hartshorne}, on voit que
 $q_\mathcal{U}:q^{-1}(\mathcal{U})\rightarrow \mathcal{U}$ est plat.
Soit alors $\mathcal{U}'$ le sous-ensemble de $\mathcal{U}$ constitué des $(t,H)$ tels
que $\mathfrak{X}_t\cap H$ est irréductible et génériquement réduit. Il est ouvert par \ref{ouv}. Par conséquent, $\mathcal{F}$, qui est le complémentaire de $\mathcal{U}'$ dans $\mathcal{H}_e$, est fermé.
 Soit $\mathcal{F}'$ la réunion des composantes irréductibles de $\mathcal{F}$ dominant $\mathbb{A}^{1}$. Comme $\mathfrak{X}$ est triviale au-dessus de $\mathbb{A}^{1}\setminus\{0\}$, c'est aussi le cas de $\mathcal{F}$ ; ainsi, $\mathcal{F}$ et $\mathcal{F}'$ coïncident au-dessus de $\mathbb{A}^{1}\setminus\{0\}$.

On va montrer que $\mathcal{F}'_0\subset \mathcal{G}_1\cup \mathcal{G}_2\cup\bigcup_i\mathcal{F}^{\rm{igr}}_e(L_i)$, où $\mathcal{G}_1$ est l'ensemble des hypersurfaces contenant $L$ et $\mathcal{G}_2$ l'ensemble des hypersurfaces dont l'intersection avec $L$ est $eS$. Admettons dans un premier temps cette inclusion. Quand $e\geq2$, $\codim_{\mathcal{H}_e}(\mathcal{F}^{\rm{igr}}_e(L_i))={e+n-1 \choose e}-n$ par \ref{caslineaire}, et $\codim_{\mathcal{H}_e}(\mathcal{G}_1\cup \mathcal{G}_2)={e+n-1 \choose e}-1$. Ainsi, $\codim_{\mathcal{H}_e}(\mathcal{F}'_0)\geq{e+n-1 \choose e}-n$. Par propreté de $p_1:\mathcal{F}'\rightarrow\mathbb{A}^{1}$, la dimension de ses fibres est semi-continue supérieurement, et il existe $t\not=0$ tel que $\codim_{\mathcal{H}_e}(\mathcal{F}'_t)\geq{e+n-1 \choose e}-n$. Mais, $\mathcal{F}'_t=\mathcal{F}_t=\mathcal{F}^{\rm{igr}}_e(\mathfrak{X}_t)$, ce qui permet de conclure car $\mathfrak{X}_t$ est projectivement équivalente à $X$.

Quand $e=1$, $\codim_{\mathcal{H}_1}(\mathcal{F}^{\rm{igr}}_1(L_i))=n+1$ et $\codim_{\mathcal{H}_1}(\mathcal{G}_1\cup \mathcal{G}_2)=n-1$. On conclut en raisonnant identiquement.

\begin{etape2}

Changement de base.

\end{etape2}

Il reste à prouver l'inclusion admise ci-dessus. On raisonne par l'absurde en choisissant $H\in \mathcal{F}'_0$ telle que $H\notin \mathcal{G}_1\cup \mathcal{G}_2\cup\bigcup_i\mathcal{F}^{\rm{igr}}_e(L_i)$. Ceci implique que $\mathfrak{X}_0\cap H$ est réduit et a exactement $d$ composantes irréductibles distinctes $H_i=L_i\cap H$ (il n'y a pas de points immergés car $\mathfrak{X}_0\cap H$ est intersection complète).

 D'une part, $H$ appartient à une composante irréductible de $\mathcal{F}$ qui domine $\mathbb{A}^{1}$, et d'autre part, comme $H\notin \mathcal{G}_1$, $(0,H)\in \mathcal{U}$. On peut donc trouver une courbe intègre $B$, un morphisme $f:B\rightarrow \mathcal{F}\cap \mathcal{U}$ tel que $p_1\circ f:B\rightarrow\mathbb{A}^{1}$ soit dominant, et un point $b\in B$ tel que $f(b)=(0,H)$. On note $\mathfrak{Z}_B\subset\mathfrak{X}_B$ le tiré en arrière de $\mathfrak{Z}\subset\mathfrak{X}\times\mathcal{H}_e$ par $f:B\rightarrow\mathbb{A}^{1}\times\mathcal{H}_e$. Comme $f$ est à valeurs dans $\mathcal{U}$, par changement de base, $q_B:\mathfrak{Z}_B\rightarrow B$ est propre et plat. 
Comme $\mathfrak{Z}_{b}=\mathfrak{X}_0\cap H$ est réduit, en appliquant \cite{EGA43} 12.2.4 (v), et quitte à restreindre $B$, on peut supposer les fibres de $q_B$ géométriquement réduites. Or, $q$ étant à valeurs dans $\mathcal{F}$, les fibres de $q_B$ ne sont pas géométriquement intègres. Elles sont donc nécessairement géométriquement réductibles. Quitte à remplacer $B$ par un revêtement, on peut alors supposer que la fibre générique de $q_B$ est réductible, c'est-à-dire que $\mathfrak{Z}_B$ est réductible. Finalement, en normalisant, on voit qu'on peut choisir $B$ lisse.

Dans la suite de la démonstration, on va obtenir une contradiction en montrant l'irréductibilité de $\mathfrak{Z}_B$.

\begin{etape3}

Irréductibilité de $\mathfrak{Z}_{B}$.

\end{etape3}

  Comme $H\notin \mathcal{G}_2$, $H$ contient un point fermé $P\in L\setminus S$. Ainsi, $P\in\mathfrak{Z}_{b}=\mathcal{X}_0\cap H$. Soit $Z$ une composante irréductible de $\mathfrak{Z}_B$ contenant $P$. Par dimension, $q_B:Z\rightarrow B$ est dominante, donc plate. 
Admettons un instant que $Z$ contienne les $d$ composantes irréductibles $H_1,\ldots, H_d$ de $\mathfrak{Z}_{b}$. Alors, si $b' \in B$,

$$\begin{array}{llll} \deg(Z_{b'})&= & \deg(Z_{b})    \textrm{\hspace{2mm}par platitude}\\

&=& \deg(\mathfrak{Z}_{b})\textrm{\hspace{2mm}car\hspace{2mm}} \mathfrak{Z}_{b}=Z_{b} \\

&=&\deg(\mathfrak{Z}_{b'})  \textrm{\hspace{2mm}par platitude.}\\

\end{array}$$

 Comme $\mathfrak{Z}_{b'}$ et $Z_{b'}$ sont de même dimension, et qu'on a une inclusion, cela implique $Z_{b'}=\mathfrak{Z}_{b'}$, soit $Z=\mathfrak{Z}$. Ce qui contredit la réductibilité de $\mathfrak{Z}$.

Il reste à prouver l'assertion admise ci-dessus : on doit montrer que $H_k\subset Z_{b}$ pour $1\leq k\leq d$. C'est ce qu'on va obtenir dans la suite, comme conséquence du théorème de Ramanujam-Samuel. Dans l'étape suivante, on introduit les variétés auxquelles on pourra appliquer ce théorème.

\begin{etape4}

Normalisation.

\end{etape4}

  Considérons la sous-variété $\Gamma$ de $\mathbb{P}^{2}\times\mathbb{A}^{1}$ d'équation $G(x_0,x_1,tx_2)=0$.
La projection depuis $S$ induit $\pi:\mathfrak{X}\setminus S\rightarrow \Gamma$, lisse car de fibres des espaces affines $\mathbb{A}^{N-2}$.
  La projection $pr_2:\Gamma\rightarrow\mathbb{A}^1$ est triviale de fibre $C$ au-dessus de $\mathbb{A}^{1}\setminus\{0\}$. Au-dessus de $0$, elle est lisse sauf en le point $\Omega=[0:0:1]$ : en effet, $pr_2^{-1}(0)$ est constitué de $d$ droites $D_1,\ldots, D_d$ s'intersectant en $\Omega$. Remarquons que $\pi^{-1}(D_k)=L_k\setminus S$.

 En tirant en arrière $\mathfrak{X}\setminus S\stackrel{\pi}{\rightarrow} \Gamma\stackrel{pr_2}{\rightarrow}\mathbb{A}^1$ par $p_1\circ f:B\rightarrow\mathbb{A}^{1}$, on obtient $\mathfrak{X}_B\setminus S_B\stackrel{\pi_B}{\rightarrow} \Gamma_B\stackrel{pr_{2,B}}{\rightarrow} B$. On identifie $\Omega$ et $D_k$ aux sous-variétés correspondantes de $(\Gamma_B)_b=\Gamma_0$, et $P$ et $L_k\setminus S$ aux sous-variétés correspondantes de $(\mathfrak{X}_B\setminus S_B)_b=(\mathfrak{X}\setminus S)_0$.

  Notons $\nu:\widetilde{\Gamma_B}\rightarrow \Gamma_B$ la normalisation de $\Gamma_B$. Comme $D_k\setminus\Omega$ est dans le lieu lisse de $\Gamma_B$, au-dessus duquel $\nu$ est un isomorphisme, on peut considérer la transformée stricte de $D_k$ dans $\widetilde{\Gamma_B}$ : on la note encore $D_k$ et on notera $i_k:D_k\rightarrow\widetilde{\Gamma_B}$ l'inclusion. Montrons que $\Omega$ a un unique antécédent $\tilde{\Omega}$ dans $\widetilde{\Gamma_B}$. D'une part, les composantes irréductibles de $(\widetilde{\Gamma_B})_b$ sont exactement les $i_k(D_k)_{1\leq k\leq d}$, de sorte que $\nu^{-1}(\Omega)$ est constitué des $i_k(\Omega)_{1\leq k\leq d}$. D'autre part, comme la fibre générique de $pr_{2,B}\circ\nu$ est la normalisation de $C$ qui est connexe, \cite{EGA43} 15.5.9 (ii) montre que $(\widetilde{\Gamma_B})_b$ est connexe. Ceci n'est possible que si $i_k(\Omega)$ ne dépend pas de $k$ : on note ce point $\tilde{\Omega}$.

On tire en arrière $\nu$ par $\pi_B$ pour obtenir $\mathfrak{W}$ muni de deux projections $\tilde{\pi}$ et $\nu'$ respectivement sur $\widetilde{\Gamma_B}$ et $\mathfrak{X}_B\setminus S_B$. Tirant en arrière $i_k$ par $\tilde{\pi}$, on obtient $i'_k:L_k\setminus S\rightarrow\mathfrak{W}$, qui est la transformée stricte de $L_k\setminus S$. Comme $\Omega$ a un unique antécédent par $\nu$, $P$ a un unique antécédent par $\nu'$, égal à $i_k'(P)$ pour tout $k$, qu'on note $\tilde{P}$.

 Le diagramme cartésien de variétés pointées ci-dessous récapitule les constructions effectuées :

$$\xymatrix{
(L_k\setminus S, P)\ar[r]^{i'_k}\ar[d]_{\pi}&(\mathfrak{W}, \tilde{P})\ar[r]^{\nu'}\ar[d]_{\tilde{\pi}}& (\mathfrak{X}_B\setminus S_B, P) \ar[r]^{p'}\ar[d]_{\pi_B}&(\mathfrak{X}\setminus S, P)\ar[d]_{\pi}          \\
(D_k,\Omega)\ar[r]^{i_k}&(\widetilde{\Gamma_B},\tilde{\Omega})\ar[r]^{\nu}& (\Gamma_B,\Omega) \ar[r]^{p}\ar[d]_{pr_{2,B}} &(\Gamma,\Omega)\ar[d]_{pr_2}                  \\  
&& (B,b)\ar[r]^{p_1\circ f}&(\mathbb{A}^1,0) }$$

\begin{etape5}

Application du théorème de Ramanujam-Samuel.

\end{etape5}

  Comme $H\notin \mathcal{G}_1$, $Z_b$ ne contient pas $L$. Ainsi, $Z$ n'est pas inclus dans le lieu où $\nu'$ n'est pas un isomorphisme, et on peut considérer sa transformée stricte $W$ dans $\mathfrak{W}$. Par propreté de $\nu'$, $\nu'(W)=Z$, et on a donc $\tilde{P}\in W$.

Remarquons que comme $Z_b$ ne contient pas $L$, $W$ ne contient pas $\tilde{\pi}^{-1}(\tilde{\Omega})$.       On peut donc appliquer le théorème de Ramanujam-Samuel sous sa forme \cite{EGA44} 21.14.3 (i) au morphisme lisse de base normale $\tilde{\pi} :\mathfrak{W}\rightarrow\widetilde{\Gamma_B}$ et au diviseur $W$ de $\mathfrak{W}$ en le point $\tilde{P}$. Ainsi, $W$ est de Cartier dans $\mathfrak{W}$ en $\tilde{P}$. En particulier, comme $W$ rencontre $L_k\setminus S$ en le point $\tilde{P}$, $W$ rencontre $L_k\setminus S$ en un diviseur de $L_k\setminus S$. Cela signifie que $Z$ rencontrait $L_k$ en un diviseur de $L_k$, nécessairement égal à $H_k$. C'est ce qu'on voulait montrer.

 \end{proof}

\section{Minoration de la codimension de $\mathcal{F}^{\rm{int}}_e(X)$}

  L'objet de cette partie est de déduire le théorème \ref{codimint} du théorème \ref{codimigr}. Comme ces énoncés sont insensibles à l'extension des scalaires, on fera l'hypothèse que $K$ est algébriquement clos. 

Il faut contrôler les points immergés qui apparaissent lorsqu'on intersecte $X$ avec une hypersurface.

\begin{lemme}\label{profun}

Soit $X$ un $K$-schéma de type fini réduit. Alors $X$
ne contient qu'un nombre fini de points de profondeur $1$ et de
codimension $\geq 2$ dans $X$.

\end{lemme}

\begin{proof}[$\mathbf{Preuve}$]

La fonction
$\coprof(x)=\dim(\mathcal{O}_{X,x})-\prof(\mathcal{O}_{X,x})$ est
semi-continue supérieurement sur $X$ par \cite{EGA43} 12.1.1 (v).

Notons $F_n$ le fermé $\{x\in X\mid \coprof(x)\geq n\}$. Par
positivité de la profondeur, les points de $F_n$ sont de codimension
$\geq n$. Si $n\geq1$, $F_n$ ne contient pas de points de codimension $
n$. En effet, un tel point vérifierait $\prof(x)=0$, et \cite{EGA41}
chap. 0, 16.4.6 (i) montre que ce serait un point immergé de $X$.

 Ainsi, pour $n\geq1$, les points de $F_n$ de hauteur $n+1$ sont des points
 génériques de composantes irréductibles de $F_n$, et ils sont donc en
 nombre fini. Les points de $X$ de profondeur $1$ et de codimension
 $\geq 2$ sont exactement la réunion de tous ces points pour $n\geq 1$, et sont donc en nombre
 fini.

\end{proof}

\begin{lemme}\label{hehe}

Soit $X\subset\mathbb{P}^N$ une sous-variété de
dimension $\geq2$. Notons $x_1,\ldots x_r,$ ses points de
codimension $\geq2$ et de profondeur $1$, qui sont en nombre fini
par le lemme \ref{profun}.

Alors si $H\in\mathcal{H}_e$ ne contient pas $X$, les points
immergés de $X\cap H$ sont exactement les $x_i$ appartenant à $H$.

\end{lemme}

 \begin{proof}[$\mathbf{Preuve}$]

En effet, par \cite{EGA41} chap. 0, 16.4.6 (i), les points immergés
sont exactement ceux de profondeur $0$ et de codimension $\geq1$. Il
suffit alors de remarquer que lors d'une section non triviale par une hypersurface,
codimension et profondeur chutent tous deux de $1$.
\end{proof}

 Montrons à présent le théorème \ref{codimint} :

 \begin{proof}[$\mathbf{Preuve \text{ }du \text{ }th\acute{e}or\grave{e}me\text{ } \ref{codimint}}$]

On considère les points de $X$ de profondeur $1$ et de codimension
$\geq2$. Notons $y_1,\ldots,y_s$ ceux qui sont fermés, et
$z_1,\ldots,z_t$ les autres. Soit $\mathcal{F}_i$ l'ensemble des hypersurfaces
contenant $y_i$ et $\mathcal{G}_i$ l'ensemble des hypersurfaces contenant
$z_i$.

 Par le lemme \ref{hehe},
 $\mathcal{F}^{\rm{int}}_e(X)=\mathcal{F}^{\rm{igr}}_e(X)\cup\bigcup_i\mathcal{G}_i\cup\bigcup_i\mathcal{F}_i$.
 Or, dans $\mathcal{H}_e$, $\mathcal{F}_i$ est un sous-espace linéaire de codimension $1$, $\mathcal{G}_i$ est de codimension $\geq e+1$ par la proposition \ref{codimhyper}, et on minore la codimension de
 $\mathcal{F}^{\rm{igr}}_e(X)$ à l'aide du théorème
 \ref{codimigr}. On en déduit les minorations voulues.

Enfin, ces minorations sont optimales comme conséquence de la preuve et du fait que les minorations du théorème \ref{codimigr} sont optimales. Plus précisément, voici des variétés réalisant les cas d'égalité :

\begin{enumerate}[(i)]
\item Cas où $X$ n'a pas de point de profondeur $1$ et de codimension $>1$ :

\begin{enumerate}
\item quand $e=1$, un cône sur une courbe plane intègre.  

\item quand $e\geq2$, un espace linéaire.

\end{enumerate}

\item Cas où $X$ n'a pas de point fermé de profondeur $1$ :

\begin{enumerate}
\item quand  $n\geq3$, une variété contenant un point de profondeur $1$ dont l'adhérence est une droite, mais pas de point fermé de profondeur $1$    .

\item quand  $n=2$ et $e\geq2$, un plan.

\item quand  $n=2$ et $e=1$, un cône sur une courbe plane intègre.     

\end{enumerate}

\item Cas où $X$ possède un point fermé de profondeur $1$ :

\begin{enumerate}

\item une variété quelconque contenant un point fermé de profondeur $1$.

\end{enumerate}
\end{enumerate}

\end{proof}

\section{Le théorème de Bertini en famille}

 Comme application des résultats précédents, montrons le théorème \ref{famillebert}.

 \begin{proof}[$\mathbf{Preuve \text{ }du \text{ }th\acute{e}or\grave{e}me\text{ } \ref{famillebert}}$]

Montrons $(ii)$ : on procède de même pour $(i)$ en utilisant le théorème \ref{codimigr} à la place du théorème \ref{codimint} $(ii)$.

Soit $\mathcal{F}$ le sous-ensemble de $V\times\mathcal{H}_e$ constitué des $(v,H)$ tels que $\mathfrak{X}_v\cap H$ n'est pas géométriquement intègre de codimension $1$ dans $\mathfrak{X}_v$. Comme dans la preuve en \ref{preuvebertinigr} du théorème \ref{bertinigr}, et en utilisant \cite{EGA43} 12.2.4 (viii) à la place de la proposition \ref{ouv}, on montre que $\mathcal{F}$ est fermé dans $V\times\mathcal{H}_e$.

Les fibres de $p_1 :
\mathcal{F}\rightarrow V$ sont les $p_1^{-1}(v)=\mathcal{F}^{\rm{int}}_e(\mathfrak{X}_v)$ et sont donc de codimension $\geq e-1$ dans $\mathcal{H}_e$
par le théorème \ref{codimint} $(ii)$. La codimension de $\mathcal{F}$ dans
$V\times\mathcal{H}_e$ est donc $\geq e-1=\dim(V)+1$. Ainsi, $p_2 :
\mathcal{F}\rightarrow\mathcal{H}_e$ n'est pas surjective. 

L'ensemble qui nous intéresse est le complémentaire de l'image de $p_2$, et est donc non vide. Comme il est constructible par le théorème de Chevalley, il contient un ouvert non vide. Quand $V$ est propre, $\mathcal{F}$ est également propre, et son image par $p_2$ est fermée, donc de complémentaire un ouvert. Enfin, si le corps $K$ est infini, toute variété non vide a un $K$-point, d'où l'existence de l'hypersurface définie sur $K$ recherchée.

\end{proof}




\begin{thebibliography}{20}
\bibitem{Debarreart1} O. Debarre, \rm{Varieties with ample cotangent bundle},
\emph{Comp. Math.} $\textbf{141}$ ($2005$), $1445-1459$.
\bibitem{EGA2} A. Grothendieck, \emph{\'{E}léments de géométrie algébrique II},
Publ. Math. IHES $\textbf{8}$ ($1961$).
\bibitem{EGA3} A. Grothendieck, \emph{\'{E}léments de géométrie algébrique III},
Publ. Math. IHES $\textbf{11}$ ($1961$).
\bibitem{EGA41} A. Grothendieck, \emph{\'{E}léments de géométrie algébrique IV 1},
Publ. Math. IHES $\textbf{20}$ ($1964$).
\bibitem{EGA42} A. Grothendieck, \emph{\'{E}léments de géométrie algébrique IV 2},
Publ. Math. IHES $\textbf{24}$ ($1965$).
\bibitem{EGA43} A. Grothendieck, \emph{\'{E}léments de géométrie algébrique IV 3},
Publ. Math. IHES $\textbf{28}$ ($1966$).
\bibitem{EGA44} A. Grothendieck, \emph{\'{E}léments de géométrie algébrique IV 4},
Publ. Math. IHES $\textbf{32}$ ($1967$).
\bibitem{Hartshorne} R. Hartshorne, \emph{Algebraic geometry},
Springer-Verlag, $1977$.
\bibitem{Jouanolou} J.-P. Jouanolou, \emph{Théorèmes de Bertini et applications},
Birkhäuser, $1983$.
\bibitem{Poonen} B. Poonen, \rm{Bertini theorems over finite fields},
\emph{Annals of math.} $\textbf{160}$ ($2004$), no. $3$, $1099-1127$.



\end{thebibliography}

\addcontentsline{toc}{section}{R\'{e}f\'{e}rences}

\end{document}